\documentclass[reqno]{amsart}

\usepackage{latexsym,amsfonts}
\usepackage{amsmath,amsthm}          
\usepackage{euscript}                
\usepackage{epsfig}
\usepackage[all]{xy}
\usepackage{mathrsfs}
\usepackage{calrsfs}

\usepackage[bookmarks,
bookmarksnumbered,%
 colorlinks=true,%
 linkcolor=blue,%
 citecolor=red,%
 filecolor=blue,%
 urlcolor=blue,%
]
{hyperref}

\theoremstyle{plain} \numberwithin{equation}{section}
\newtheorem{main}{Theorem}

\newtheorem{thm}{Theorem}[section]
\newtheorem{cor}[thm]{Corollary}
\newtheorem{prop}[thm]{Proposition}

\newtheorem{lemma}[thm]{Lemma}

\theoremstyle{definition}
\newtheorem{remark}{Remark}[section]

\newtheorem{defn}[remark]{Definition}

\newtheorem{rmk}[thm]{Remark}

\newcommand{\bi}{\begin{itemize}}
\newcommand{\ei}{\end{itemize}}
\newcommand{\bp}{\begin{proof}}
\newcommand{\ep}{\end{proof}}

\DeclareMathOperator{\codim}{codim}

\def\dim{\mbox{dim}}
\def\ra{\rightarrow}
\def\cal{\mathcal}

\def\PP{\mathbb{P}}
\def\QQ{\mathbb{Q}}

\def\RR{\mathbb{R}}

\def\OO{\cal O}

\def\s-{\setminus}

\def\cpb{\overline{\mathbb C \mathbb P ^2}}

\begin{document}

\title{Balanced metrics on uniruled manifolds}

\author[Chiose]{Ionu\c{t} Chiose}

\address{
	Institute of Mathematics of the Romanian Academy,  P.O. Box 1-764, Bucharest 014700,  Romania}
	
	\email{Ionut.Chiose@imar.ro}

\author[R\u asdeaconu]{Rare\c s R\u asdeaconu}

\address{        
        Department of Mathematics, 1326 Stevenson Center, Vanderbilt University, Nashville, TN, 37240, USA}
        
        \email{rares.rasdeaconu@vanderbilt.edu}

\author[\c Suvaina]{Ioana \c Suvaina}

\address{        
        Department of Mathematics, 1326 Stevenson Center, Vanderbilt University, Nashville, TN, 37240, USA}

\email{ioana.suvaina@vanderbilt.edu}

\date{\today}



\begin{abstract}
We show that an $n-$dimensional Moishezon manifold 
is uniruled if and only if it supports a balanced metric 
$\omega^{n-1}$ of positive total scalar Chern curvature. 
A similar statement also holds true for class $\cal C$ 
manifolds of dimension three. 
\end{abstract}

\maketitle

\thispagestyle{empty}

\tableofcontents

\section*{Introduction}

A compact complex manifold $M$ is called uniruled 
if there exists a rational curve passing through every 
point of $M.$  A differential geometric characterization 
of uniruledness in complex dimension two was given
by Yau \cite{yau}. He proved that a K\"ahler surface $S$ 
has Kodaira dimension $-\infty$ (equivalently, uniruled) 
if and only if it admits a K\"ahler metric $\omega$ of 
positive total scalar curvature. This is equivalent to
\begin{equation}
\label{cond-surfaces}
\int_S c_1(K_S)\wedge \omega<0,
\end{equation}
where $K_S$ denotes the canonical line bundle of $S.$ 

\medskip

The aim of this article is to extend Yau's differential 
geometric characterization in higher dimensions. 
In one direction, the existence of a K\"ahler metric of positive total 
scalar curvature on projective uniruled manifolds has been 
recently discussed by Heier and Wong \cite[Section 5]{heier-wong}, 
but a definite conclusion is elusive. Such metrics are known to exist 
on some uniruled manifolds. Most notably, they exist on projective 
Mori fiber spaces of dimension three, as established by Demailly, 
Peternell and Schneider \cite[Proposition 4.9]{dps}.  An approach to 
this existence question, which indicates that in general the answer 
is negative, is proposed by the second author  in the case of rationally 
connected threefolds \cite{ratcon}. This suggests that  instead of 
searching for K\"ahler metrics of positive total scalar curvature on 
uniruled manifolds, one should broaden the search to a larger class 
of metrics. To detect a suitable such   class of Hermitian metrics we 
follow Yau's original proof \cite{yau}.
Yau's approach to find K\"ahler metrics of positive total scalar 
curvature on uniruled surfaces relies on the minimal model theory. 
His proof follows in two steps: 
\begin{itemize}
\item[ A)] Bimeromorphic invariance: Yau shows that the existence 
of such metrics is an invariant property under bimeromorphic maps. 
In the case of surfaces, the invariance under blow-ups suffices.

\item[ B)] Existence of a K\"ahler metric of positive 
total scalar Chern curvature on an exhaustive list of  
bimeromorphism classes of uniruled surfaces: Yau 
proved the existence of K\"ahler metrics satisfying 
(\ref{cond-surfaces}) on all geometrically ruled surfaces. 
\end{itemize}

\medskip

To extend Step A in higher dimensions, recall that any 
bimeromorphic map decomposes by the weak factorization 
theorem \cite[Theorem 0.3.1]{wft} into a sequence of blow-ups 
and blow-downs with smooth centers. A well-known fact is that, 
unlike uniruledness,  
the class of K\"ahler manifolds of dimension greater than or equal 
to three is not closed under bimeromorphisms. We are led to 
consider a larger class of manifolds which is invariant under 
bimeromorphisms. From the work of Alessandrini and 
Bassanelli \cite{alessandrini3, alessandrini2}, it is known 
that the class of manifolds carrying balanced metrics, i.e., 
Hermitian metrics with co-closed K\"ahler form (see 
\cite{michelsohn} and Section \ref{metrics-intro}), satisfies 
this property. In dimension two, any balanced metric 
is in fact K\"ahler, but in higher dimensions there exist 
non-K\"ahler manifolds which admit balanced metrics 
or K\"ahler manifolds which admit non-K\"ahler balanced 
metrics. We prove:

\begin{main}
\label{bi-invariance}
Let $X$ and $Y$ be two bimeromorphic compact complex 
manifolds of dimension $n$. If there exists a balanced 
metric $\omega^{n-1}_X$ on $X$ such that 
$$
\int_X c_1(K_X)\wedge\omega_X^{n-1}<0,
$$
then there exists 
a balanced metric $\omega_Y^{n-1}$  on $Y$ such that 
$$
\int_Yc_1(K_Y)\wedge \omega_Y^{n-1}<0.
$$
\end{main}

Demailly, Peternell and Schneider also asked if Step A can be 
accomplished for normal projective varieties \cite[Problem 4.12]{dps}. 
Theorem \ref{bi-invariance} gives a  partial answer to their question.

\medskip

An extension of Step B to uniruled manifolds of higher dimensions 
relies on the state of the art of the minimal model program. For 
projective uniruled  manifolds one can find a bimeromorphic simpler 
model in any dimension \cite{bchm}. These bimeromorphic models 
are higher dimensional analogs of the geometrically ruled surfaces, 
called Mori fiber spaces (see Section \ref{mfs}). 
We show that every Mori fiber space admits K\"ahler metrics of 
positive total scalar curvature, and we obtain:

\begin{main}
\label{thmA-proj}
Every $n$-dimensional, Moishezon, uniruled manifold $X$  admits a 
balanced metric $\omega^{n-1}$ such that 
$$
\int_X c_1(K_X)\wedge \omega^{n-1}<0.
$$
\end{main}
Recall that a compact complex manifold is Moishezon if it is bimeromorphic 
to a projective manifold.

We provide two proofs for this result. One proof uses the minimal 
model program. A second proof is based on ideas of 
Toma \cite{toma}, and it relies on the results of Boucksom, Demailly, 
P\u aun and Peternell \cite{boucksom}, bypassing the minimal 
model program.

\medskip

A generalization of the minimal model program to the class of 
K\"ahler manifolds is known only in complex dimension three 
\cite{hor-pet1, hor-pet2}.  We prove the following extension 
of Theorem \ref{thmA-proj} in dimension three:

\begin{main}
\label{thmA}
Every uniruled threefold $X$ of class $\cal C$ admits a balanced metric 
$\omega^2$ such that 
$$
\int_X c_1(K_X)\wedge \omega^2<0.
$$
\end{main}
Recall that a complex manifold is called of class $\cal C$ if 
it is bimeromorphic to a K\"ahler manifold. This class 
of manifolds is strictly larger than the class of K\"ahler 
manifolds in dimension three or more, and it contains 
the class of Moishezon manifolds. Every class $\cal C$
manifold carries balanced metrics by 
\cite{alessandrini3, alessandrini2}.

\medskip

The bimeromorphism invariance of the class of balanced manifolds indicates 
that balanced metrics are natural to be considered as good replacements 
of K\"ahler metrics in order to extend Yau's differential geometric characterization of 
uniruledness in higher dimensions. However, this is not the only class of 
Hermitian metrics with such good properties. In fact,  every complex manifold 
admits Gauduchon metrics, that is positive $(1,1)$-forms $\omega$ such 
that $\partial\bar\partial \omega^{n-1}=0$ \cite{g-metrics}. 
Notice that every balanced metric is a Gauduchon metric, while the converse is false. 
Moreover, from the positivity criterion of Lamari \cite{lamari1} and 
\cite[Corollary 0.3]{boucksom} one can see that every uniruled projective 
manifold admits Gauduchon metrics of positive total scalar Chern 
curvature (see also Theorem \ref{characterization} below). In Theorems 
\ref{thmA-proj} and \ref{thmA} we prove therefore a stronger result.

\bigskip

Conversely, Yau's approach \cite{yau} can be adapted to show that the existence of a 
K\"ahler or a balanced metric of positive total scalar Chern curvature 
on a complex manifold implies that the Kodaira dimension of the manifold 
is $-\infty$. One can easily see that uniruledness implies that Kodaira 
dimension is $-\infty$, but the converse is a well-known open problem. 
Heier and Wong were able to show in \cite[Theorem 1.1]{heier-wong} that every projective 
manifold which admits a K\"ahler metric of positive total scalar curvature is in 
fact uniruled. We extend here Theorem 1.1 of Heier and Wong 
\cite{heier-wong}, and combining with the results from Theorems \ref{thmA-proj} 
and \ref{thmA} we provide the following characterization of uniruledness:

\begin{main}
\label{characterization}
Let $X$ be an $n$-dimensional Moishezon manifold. The following 
statements are equivalent:
\begin{itemize}
\item[ i)] $K_X$ is not pseudoeffective;
\item[ ii)] $X$ is uniruled;
\item[ iii)] $X$ admits a balanced metric of positive Chern total scalar 
curvature;
\item[ iv)] $X$ admits a Gauduchon metric of positive Chern total scalar 
curvature.
\end{itemize}
Moreover, the same statements hold true if $n=3$ and $X$ is of 
Fujiki class $\mathcal C.$
\end{main}

The proof of the implications $ iv)\Longrightarrow i)\Longrightarrow ii)$ 
relies on the positivity criterion of Lamari \cite[Th\'eor\`eme 1.2 (1)]{lamari1}, 
and on remarkable results of  Boucksom, Demailly, Peternell and P\u aun  
\cite{boucksom} and Brunella \cite{brunella}.

\medskip

We explore next the possibility of extending 
the above characterization of uniruledness in terms of the 
positivity of the total scalar Chern curvature of a balanced 
metric beyond class $\cal C.$ 
In general, the existence of a balanced metric fails. 
However, in dimension three, a large class of 
uniruled manifolds admitting such metrics is given by complex 
manifolds bimeromorphic to  twistor spaces \cite{ahs}. We prove:

\begin{main}
\label{twistors}
Every three dimensional complex manifold $X$ bimeromorphic to a twistor space 
admits a balanced metric $\omega^2$ such that 
$$
\int_{X} c_1(K_X)\wedge \omega^2<0.
$$
\end{main}

\section{Total scalar curvatures}
\label{metrics-intro}

In this section we briefly recall some well-known background 
material in complex differential geometry to introduce the 
terminology.

\bigskip

Let $(M,g)$ be a Hermitian manifold and $\omega$ its
K\"ahler form. On  $(M,g)$ one can consider two 
canonical connections: the Levi-Civita connection, and 
the Chern connection.

Let $s$ denote the scalar 
curvature of the Levi-Civita connection. The {\em total scalar 
Riemannian curvature} is defined as 
$$
\int_M s\mu_g=\int_M  \frac{s\omega^n}{n!},
$$ 
where $\mu_g=\dfrac{\omega^n}{n!}$ is the volume form.

Let $s_C$ denote the scalar curvature of the Chern 
connection associated to the Hermitian metric $g.$ The 
{\em total scalar Chern curvature} is defined by 
$$
\int_M s_C\mu_g. 
$$ 
The Ricci curvature form of the Chern connection represents
the first Chern class of $M$ rescaled by a factor of $2\pi,$ 
and $c_1(M)=-c_1(K_M),$ where $K_M$ 
is the canonical line bundle of $M.$
Since the scalar curvature is the trace of the Ricci curvature form, 
we can write 
\begin{equation}
\label{tscc}
\int_M s_C\mu_g=\int_M  \frac{s_C\omega^n}{n!}=
-\frac{2\pi}{(n-1)!}\int_M c_1(K_M)\wedge \omega^{n-1}.
\end{equation}

A result due to Gauduchon \cite[page 506]{g-torsion} 
(see also \cite[Corollary 1.11]{liu-yang}) compares the total 
scalar Riemannian curvature and the total scalar Chern curvature:

\begin{prop}
\label{ch>lc}
Let $(M,g)$ be a compact, complex manifold equipped with a 
Hermitian metric. Then
$$
\int_M s_C\mu_g\geq \frac12\int_M s\mu_g,
$$
with equality if and only if the metric is K\"ahler.
\end{prop}

\begin{cor}
\label{ineq}
Let $(M,g)$ be a compact, complex manifold of dimension $n$ 
equipped with a Hermitian metric. If the scalar Riemannian 
curvature of $M$ is positive, then 
$$
\int_M c_1(K_M)\wedge \omega^{n-1}<0.
$$
\end{cor}
\qed

\begin{defn}
Let $(M,g)$ be a compact complex manifold of 
complex dimension $n$ equipped with a Hermitian
metric $g,$ and let $\omega$ denote its K\"ahler form.
If $d\omega = 0,$ then $g$ is called a K\"ahler metric. 
A complex manifold which admits a K\"ahler metric is 
called a K\"ahler manifold.
\end{defn}

If $g$ is a K\"ahler metric, then its K\"ahler form 
$\omega$ is a real, $d$-closed, strictly positive $(1,1)$-form. 
Conversely, given a smooth, strictly positive, $d$-closed $(1,1)$-form 
$\omega$, there exists a Hermitian metric $g$ whose K\"ahler 
form is $\omega.$ We will use the notation $(M,\omega)$ to 
denote a K\"ahler manifold with prescribed K\"ahler form.

\begin{defn}
Let $(M,g)$ be a compact complex manifold of 
complex dimension $n$ equipped with a Hermitian
metric $g,$ and let $\omega$ denote its K\"ahler form.
If $d(\omega^{n-1}) = 0,$ then $g$ is called a balanced 
metric. A complex manifold which admits a balanced 
metric is called a balanced manifold. We will use the 
notation $(M,\omega^{n-1})$ to denote a balanced manifold. 
\end{defn}

Given a balanced metric of K\"ahler form $\omega,$ the 
$(n-1,n-1)$-form $\omega^{n-1}$ is real, strictly positive 
and $d$-closed. Conversely, it is an easy exercise in linear 
algebra to see that given a real, strictly positive, $d$-closed 
$(n-1,n-1)$-form $\Omega,$ there exists a unique Hermitian 
metric of K\"ahler form $\omega$ such that $\Omega=\omega^{n-1}$ 
(\cite[page 279]{michelsohn}). Throughout the paper, by 
a balanced metric we mean a real, $d$-closed, strictly positive 
$(n-1,n-1)$-form, denoted by $\omega^{n-1}.$

\bigskip

A K\"ahler manifold is balanced, and if $n=2$ the converse is also true. 
In higher dimensions the converse is false. A large class 
of counterexamples is provided by the twistor spaces of closed 
anti-self-dual four-manifolds (see Sect. 4). 
Another interesting class of non-K\"ahler balanced manifolds has 
been found by Fu, Li and Yau. In \cite{fuliyau}, the authors 
showed that the complex structures with trivial canonical bundles constructed  
by Lu and Tian \cite{lu-tian} and Friedman \cite{friedman} 
on connected sums of  $S^3\times S^3$ carry a balanced metric.

\section{Positive cones in Bott-Chern and Aeppli cohomology groups}
\label{cones}

In this section we recall the definitions of the Bott-Chern and Aeppli 
cohomology groups, and of the pseudoeffective and the 
nef cones. In the K\"ahler case, these cohomology groups 
are isomorphic to the usual Dolbeault cohomology groups 
due to the $\partial\bar\partial$-lemma. However, we 
prefer to work with the Bott-Chern and Aeppli cohomology 
groups since the class of a $d$- or $i\partial\bar\partial$-closed 
positive current lies naturally in these cohomology groups, and, 
moreover, the duality statements between the nef and 
pseudoeffective cones (Theorem \ref{duality}) can be naturally 
stated in this setting. For more details, see \cite{schweitzer}.

\bigskip

Let $X$ be a compact complex manifold of dimension $n.$ 
The Bott-Chern cohomology groups are defined as 
\begin{equation*}
H^{p,q}_{BC}(X,{\mathbb C})=
\frac{\{\alpha\in {\mathcal C}^{\infty}_{p,q}(X)\vert d\alpha=0\}}
{\{i\partial\bar\partial\beta\vert \beta\in {\mathcal C}^{\infty}_{p-1,q-1}(X)\}},
\end{equation*}
and the Aeppli cohomology groups are
\begin{equation*}
H^{p,q}_A(X,{\mathbb C})=
\frac{\{\alpha\in {\mathcal C}^{\infty}_{p,q}(X)\vert i\partial\bar\partial \alpha=0\}}
{\{\partial\beta+\bar\partial\gamma\vert\beta\in {\mathcal C}^{\infty}_{p-1,q}(X),
\gamma\in {\mathcal C}^{\infty}_{p,q-1}(X)\}}
\end{equation*} 
Since all the operators involved in the definitions of the 
above cohomology groups are real in bidegrees $(p,p)$ 
the real cohomology groups $H^{p,p}_{BC}(X,{\mathbb R})$ and 
$H^{p,p}_A(X,{\mathbb R})$ are well-defined. 
The above groups can be defined by using smooth 
forms or currents. We use the notation $[s]$ for the class of a 
$d$-closed form or current $s$ in 
$H^{\bullet,\bullet}_{BC}$ and $\{t\}$ for the class of a 
$\partial\bar\partial$-closed form or current $t$ in 
$H^{\bullet,\bullet}_A.$ The groups 
$H^{p,q}_{BC}(X, {\mathbb C})$ and 
$H^{n-p,n-q}_A(X, {\mathbb C})$ are dual via the pairing
\begin{equation} 
H^{p,q}_{BC}(X, {\mathbb C})\times H_A^{n-p,n-q}
(X, {\mathbb C})\to {\mathbb C}, 
([\alpha ],\{\beta\})\to 
\int_X\alpha\wedge\beta
\end{equation}
By an abuse of notation, we also denote by $(\alpha,\beta)$ 
the evaluation $\int_X \alpha\wedge \beta,$ regardless of 
whether $\alpha$ and $\beta$ denote appropriate forms, currents or 
cohomology classes.

\begin{defn}[Lelong \cite{lelong}]
Let $T$ be a current of bi-dimension $(p,p).$ We say that 
$T$ is a positive current, and we write $T\geq 0,$ if 
$T\wedge i\alpha_1\wedge \bar \alpha_1 \wedge \cdots 
\wedge  i\alpha_p\wedge \bar \alpha_p$ is a positive measure, for all 
 smooth $(1,0)-$forms $\alpha_1,\dots, \alpha_p.$
\end{defn}

For $\#\in\{BC, A\}$ and $p\in\{1,n-1\}$ we define the following cones:
\begin{enumerate}
\item the $\#-${\it pseudoeffective} cone 
\begin{equation}
\label{pef}
{\mathcal E}^p_{X,\#}=\{\gamma\in H^{p,p}_{\#}
(X,{\mathbb R})\vert\exists T\geq 0, T\in\gamma\},
\end{equation}
where by $T$ we denote here a current.

\item the $\#-${\it nef} cone 
\begin{equation}
\label{nef}
{\mathcal N}^p_{X,\#}=\{\gamma\in H^{p,p}_{\#}
(X,{\mathbb R})\vert \forall \varepsilon >0,
\exists\alpha_{\varepsilon}\in \gamma,
\alpha_{\varepsilon}\geq -\varepsilon \omega^p\}
\end{equation}
where $\omega$ is the K\"ahler form of a fixed Hermitian 
metric on $X$ and $\alpha_{\varepsilon}$ denotes a smooth 
$(p,p)-$form. 
\end{enumerate}

Notice that all of the cones defined above are convex cones.

\begin{rmk}
The pseudoeffective and nef cones ${\mathcal E}^1_{X,BC}$ 
and ${\mathcal N}^1_{X,BC}$ were first introduced by Demailly 
\cite[Definition 1.3]{de-reg}, who stressed their importance. We
adapt here his definitions to $(n-1,n-1)$ Bott-Chern cohomology 
classes and to $(p,p)$ Aeppli cohomology classes, where 
$p\in\{1,n-1\}$. 
\end{rmk}

\medskip

The following two lemmas are standard, and some of the statements below 
are proved in \cite[Proposition 6.1]{de-reg}. 
As they play a crucial part in our argument, and for the reader's convenience, 
we include their proofs.

\begin{lemma}
\label{E1BC-closed}
The cone ${\mathcal E}_{X,BC}^1$ is closed and 
${\mathcal N}_{X,BC}^1\subset {\mathcal E}_{X,BC}^1.$ 
\end{lemma}

\begin{proof} The proof of the lemma relies on the 
existence of Gauduchon metrics on any compact complex manifold 
\cite{g-metrics}. That means $X$ admits a Hermitian metric $g$ 
with K\"ahler form $\omega$ satisfying 
$\partial\bar\partial \omega ^{n-1}=0.$

Indeed, suppose $([T_j])_j$ is a 
sequence of pseudoeffective classes represented by the closed 
positive currents $T_j$ such that 
$[T_j]\to\gamma\in H^{1,1}_{BC}(X, {\mathbb R})$. 
Fix $g$ a Gauduchon metric with K\"ahler form $\omega,$ 
and notice that $\int_XT_j\wedge \omega^{n-1}$ depends 
only on the Aeppli cohomology class $\{\omega^{n-1}\}$ 
and on the BC-cohomology class $[T_j]$, not on the representative 
$\omega$. Since the sequence $(\int_X T_j\wedge \omega^{n-1})_j$ is bounded, 
we can assume, after passing to a subsequence,  that 
$(T_j)_j$ is weakly convergent to a closed positive current 
$T$. Then $\gamma=[T]\in {\mathcal E}_{X,BC}^1.$

To prove that ${\mathcal N}_{X,BC}^1\subset {\mathcal E}_{X,BC}^1,$ let 
$[\alpha]\in {\mathcal N}_{X,BC}^1,$ where $\alpha$ is a $d$-closed smooth 
$(1,1)$-form. Then, by definition, for every $\varepsilon >0$, there exists 
$\varphi_{\varepsilon}\in {\mathcal C}^{\infty}(X,{\mathbb R})$ such that 
$\alpha_{\varepsilon}:=
\varepsilon \omega+\alpha+i\partial\bar\partial\varphi_{\varepsilon}\geq 0$. 
Since $\int_X\alpha_{\varepsilon}\wedge \omega^{n-1}$ is bounded for 
$0<\varepsilon\leq 1$, we extract a weakly convergent subsequence 
$(\alpha_{\varepsilon_j})_j$, converging to a closed, positive current in 
class $[\alpha]$. Hence $[\alpha]\in {\mathcal E}_{X,BC}^1.$
\end{proof}

 \begin{lemma}
The cones ${\mathcal N}^p_{X,\#}$ are closed, where 
$p\in\{1,n-1\}$ and $\#\in \{BC, A\}$.
\end{lemma}

\begin{proof}
Let $\{\gamma_j\}_j$ be a sequence, where 
$\gamma_j\in {\mathcal N}_{X,\#}^p$ and 
$\gamma_j\to \gamma$ in $H^{p,p}_{\#}(X,\RR)$. 
In each cohomology 
class $\gamma_j,$ we choose the unique harmonic representative 
$\beta_j$ and let $\beta$ be the unique harmonic representative 
in $\gamma$ with respect to some fixed Hermitian metric on $X$ 
(see \cite{schweitzer} for more on the harmonic forms 
in the Bott-Chern and Aeppli cohomology groups 
\footnote{The Bott-Chern Laplacian was
introduced by Kodaira and Spencer in \cite[page 71]{ks}. In {\em op. cit.}, Schweitzer
adapted this construction to define the Aeppli Laplacian on the same
model.}). Then, from 
the standard theory of elliptic operators, it follows that 
$\beta_j\to\beta$ in the ${\mathcal C}^{\infty}$ topology. 
This immediately implies that $\gamma$ is nef. 
Indeed, for every $p\in\{1,n-1\}$, given $\varepsilon>0$, 
we can find $j_{\varepsilon}$ such that 
$\beta-\beta_{j_{\varepsilon}}\geq-\frac {\varepsilon}{2}\omega^p$. 
Since $\gamma_{j_{\varepsilon}}$ (which is the class of $\beta_{j_{\varepsilon}}$) 
is nef, it follows that for every $\delta>0$ there exists  a smooth form 
$\lambda_{\varepsilon,\delta}\in \gamma_{j_{\varepsilon}}$  such that 
$\lambda_{\varepsilon,\delta}\geq-\frac{\delta}{2}\omega^p$. 
Notice now that, for every 
$\varepsilon>0$ and $\delta>0,\, \beta-\beta_{j_{\varepsilon}}+\lambda_{\varepsilon,\delta}$ 
is a smooth representative of $\gamma$ which is 
$\geq-\frac{\varepsilon+\delta}{2}\omega^p$. Therefore $\gamma$ is nef.   
\end{proof}

\bigskip

Given $V$ a real vector space, denote by $V^*$ its dual. 
If $C$ a convex cone in $V$, we denote by $C^*\subset V^*$ 
its dual: 
$$
C^*=\{v^*\in V^*\vert v^*(c)\geq 0,\forall c\in C\}.
$$ 
By the Hahn-Banach Theorem, we have $C^{**}=\overline C$.

\begin{thm}
\label{duality}
Let $X$ be a compact complex manifold of dimension $n$. Then
\begin{itemize}

\item[i)]
${\mathcal N}_{X,BC}^1=({\mathcal E}_{X,A}^{n-1})^*,$

\item[ii)]
${\mathcal N}_{X,A}^{n-1}=({\mathcal E}_{X,BC}^1)^*.$
\end{itemize}
Moreover, if $X$ is  balanced, then
\begin{itemize}
\item[iii)]
${\mathcal N}_{X,A}^1=({\mathcal E}_{X,BC}^{n-1})^*,$

\item[iv)]
${\mathcal N}_{X,BC}^{n-1}=({\mathcal E}_{X,A}^1)^*.$
 
 \end{itemize}
 \end{thm}

\begin{proof}  The proof of the above statements either follows directly 
from \cite{lamari1}, or the arguments in \cite{lamari1} go through 
{\it{mutatis mutandis}}. For the convenience of the reader we include 
the details in the cases ii), iii) and iv) which are not covered by the results 
in \cite{lamari1}.
 
 \begin{itemize}
\item[  i)] This is Th\'eor\`eme 1.2 (1) in \cite{lamari1}.
\begin{footnote}{The cones ${\mathcal N}_{X,BC}^1$ and 
${\mathcal E}_{X,A}^{n-1}$ are the denoted by 
$P^1_{\rm nef}(X)$ and $\Pi^{n-1},$ respectively in  
\cite[Th\'eor\`eme 1.2 (1)]{lamari1}.}
\end{footnote}

\medskip

\item[ ii)] Clearly 
${\mathcal N}_{X,A}^{n-1}\subset ({\mathcal E}_{X,BC}^1)^*.$ 
Conversely, 
$({\mathcal E}_{X,BC}^1)^*\subset {\mathcal N}_{X,A}^{n-1}$ 
is equivalent to 
$({\mathcal N}_{X,A}^{n-1})^*\subset {\mathcal E}_{X,BC}^1$ 
since ${\mathcal E}_{X,BC}^1$ is closed. 
Let $[\eta]\in H^{1,1}_{BC}(X, {\mathbb R})$ be such 
that $([\eta],\gamma)\geq 0$, $\forall\gamma\in {\mathcal N}_{X,A}^{n-1}$. 
In particular, $(\eta, \Omega)\geq 0$ for any positive 
$i\partial\bar\partial$-closed $(n-1,n-1)$ form $\Omega$ on $X$. 
Lemme 1.4 in \cite{lamari1} implies the existence of a distribution $\chi$ 
such that $\eta+i\partial\bar\partial\chi\geq 0$, that is $[\eta]\in {\mathcal E}_{X,BC}^1$.

\medskip

\item[ iii)] 
The inclusion  
${\mathcal N}_{X,A}^1\subset({\mathcal E}_{X,BC}^{n-1})^*$
follows immediately.
For the opposite inclusion, we adapt the proof of 
Lemme 1.3 in \cite{lamari1} to our situation. 

\medskip

Let $\{\eta\}\in H^{1,1}_A(X, {\mathbb R})$ 
be an Aeppli cohomology class such that 
$(\{\eta\},\gamma)\geq 0$, $\forall \gamma\in {\mathcal E}_{X,BC}^{n-1}$, 
and $\eta\in {\mathcal C}^{\infty}_{1,1}(X, {\mathbb R})$ a representative.

\smallskip

We proceed by fixing a Hermitian metric on $X,$ with K\"ahler form $\phi.$
Let ${\mathcal D}'^{n-1, n-1}(X, {\mathbb R})$ denote the space of 
real currents of bidegree $(n-1,n-1)$ on $X,$ and define 
$$
C^{n-1}=\{T\in {\mathcal D}'^{n-1,n-1}(X, {\mathbb R})\vert T\geq 0, (T,\phi)=1\},
$$
which is a convex, compact set.

The set $V$ of all balanced metrics on $X$ is an open convex cone in 
$$
E=\{\lambda^{n-1}\in {\mathcal C}^{\infty}_{n-1,n-1}(X, {\mathbb R})\vert d\lambda^{n-1}=0\}.$$ 
We have $(\eta,\omega^{n-1})\geq 0$, $\forall\omega^{n-1}\in V$.
If $(\eta, \omega^{n-1})=0$, $\forall \omega^{n-1}\in V$, then 
$(\eta,\lambda^{n-1})=0$, $\forall\lambda^{n-1}\in E$ since $V$ is open in $E.$ 
From the duality between $H^{n-1,n-1}_{BC}(X, {\mathbb R})$ and $H^{1,1}_A(X, {\mathbb R})$ 
it follows that $\{\eta\}=0\in {\mathcal N}^1_{X,A}$.
We can therefore suppose that there exists $\omega_0^{n-1}\in V$ a balanced metric such that 
$(\eta, \omega_0)>0$. 
Let $D^{n-1}=C^{n-1}\cap E'$, where 
$$
E'=\{T\in {\mathcal D}'^{n-1,n-1}(X, {\mathbb R})\vert dT=0\}.$$ 
It is a convex, compact subset of ${\mathcal D}'^{n-1,n-1}(X, {\mathbb R})$ 
which is non-empty, as it contains the balanced metrics. Without loss of 
generality, we can assume that $\omega_0^{n-1}\in D^{n-1}$, i.e., that 
$(\omega_0^{n-1}, \phi)=1$. 

For $\varepsilon>0$, set 
$C(\varepsilon)=C^{n-1}+\varepsilon \omega_0^{n-1}$ and 
$D(\varepsilon)=D^{n-1}+\varepsilon \omega_0^{n-1}$. As $\omega_0^{n-1}$ is $d$-closed, we have 
$C(\varepsilon)\cap E'=D(\varepsilon)$.
Since $(\eta, T)\geq 0$, $\forall T\in D^{n-1}$ and 
$(\eta, \omega_0^{n-1})>0$, it follows that $(\eta, T)>0$, $\forall T\in D(\varepsilon)$. 
The subspace 
$$
F=E'\cap \{T\in {\mathcal D}'^{n-1,n-1}(X, {\mathbb R})\vert (\eta, T)=0\}
$$ 
is closed in ${\mathcal D}'^{n-1,n-1}(X, {\mathbb R})$ and of codimension $1$ in $E'$. Moreover, 
\begin{align*}
C(\varepsilon)\cap F=&~C(\varepsilon)\cap E'\cap \{T\in {\mathcal D}'^{n-1,n-1}(X, {\mathbb R})\vert (\eta, T)=0\}\\
=&~D(\varepsilon)\cap \{T\in {\mathcal D}'^{n-1,n-1}(X, {\mathbb R})\vert (\eta, T)=0\}\\
=&\emptyset.
\end{align*}
We can therefore separate $C(\varepsilon)$ and $F$ with a smooth $(1,1)$ form 
$\beta_{\varepsilon}$ which vanishes on $F$ and is strictly positive 
on $C(\varepsilon)$. If we let  
$\displaystyle \lambda_{\varepsilon}=
\frac {(\eta,\omega_0^{n-1})}{(\beta_{\varepsilon},\omega_0^{n-1})}$, then  
the $(1,1)$-form $\eta-\lambda_{\varepsilon}\beta_{\varepsilon}$ is zero on $E'.$
Therefore, from the duality between 
$H^{1,1}_A(X, {\mathbb R})$ and $H^{n-1,n-1}_{BC}(X, {\mathbb R})$, 
it follows that there exists $\gamma_{\varepsilon}$ a smooth $(1,0)$-form such that 
$$
\eta-\lambda_{\varepsilon}\beta_{\varepsilon}=
-\bar\partial\gamma_{\varepsilon}-\partial\bar\gamma_{\varepsilon}
$$ 
and the $(1,1)$-form 
$$
\eta+\bar\partial\gamma_{\varepsilon}+
\partial\bar\gamma_{\varepsilon}=\lambda_{\varepsilon}\beta_{\varepsilon}
$$ 
is strictly positive on $C(\varepsilon)$. If $T\in C^{n-1}$, then 
$T+\varepsilon\omega_0^{n-1}\in C(\varepsilon)$ and so
\begin{equation*}
(\eta+\bar\partial\gamma_{\varepsilon}+\partial\bar\gamma_{\varepsilon}, 
T+\varepsilon\omega_0^{n-1})=(\eta+\bar\partial\gamma_{\varepsilon}+
\partial\bar\gamma_{\varepsilon},T)+\varepsilon(\eta,\omega_0^{n-1})>0. 
\end{equation*}
Hence 
$(\eta+\bar\partial\gamma_{\varepsilon}+\partial\bar\gamma_{\varepsilon}, T)
> -\varepsilon (\eta,\omega_0^{n-1})$, $\forall T\in C^{n-1}$. 

Set now $m=(\eta,\omega_0^{n-1})$. If $T$ is a positive non-zero current of bidegree 
$(n-1,n-1)$ on $X$, then $\frac{1}{(T,\phi)}T\in C^{n-1}$, therefore 
$$(\eta+\bar\partial\gamma_{\varepsilon}+
\partial\bar\gamma_{\varepsilon},T)\geq -\varepsilon m (T,\phi),\, 
\forall T\geq 0
$$ 
which means 
$\eta+\bar\partial\gamma_{\varepsilon}+\partial\bar\gamma_{\varepsilon}
\geq -\varepsilon m \phi$. 
This implies that $\{\eta\}\in {\mathcal N}_{X,A}^1$.
 
\medskip
 
\item[ iv)] If $X$ is balanced, then 
${\mathcal E}_{X,A}^1$ 
is closed (see Lemma \ref{kah-bal-cones} below).
Clearly ${\mathcal N}_{X,BC}^{n-1}\subset ({\mathcal E}_{X,A}^1)^*$ 
and the other inclusion is equivalent to 
$({\mathcal N}_{X,BC}^{n-1})^*\subset {\mathcal E}_{X,A}^1$ since 
${\mathcal E}_{X,A}^1$ is closed. 
We adapt the proof of Lemme 1.4 in  \cite{lamari1} to our situation. 

Let $\{\theta\}\in H^{1,1}_A(X, {\mathbb R})$
be an Aeppli cohomology class such that 
$(\{\theta\},\gamma)\geq 0$, $\forall\gamma\in {\mathcal N}_{X,BC}^{n-1}$, 
and $\theta\in {\mathcal C}^{\infty}_{1,1}(X, {\mathbb R})$ a representative.
In particular, $(\theta,\omega^{n-1})\geq 0$, $\forall \omega^{n-1}\in V$, 
where $V\subset {\mathcal C}^{\infty}_{n-1,n-1}(X, {\mathbb R})$ 
is the cone of all balanced metrics on $X$. Assume there exists 
$\omega_0^{n-1}\in V$ such that 
$(\theta,\omega_0^{n-1})=0$. Let 
$\beta^{n-1}\in {\mathcal C}^{\infty}_{n-1,n-1}(X, {\mathbb R})$ 
be a $d$-closed $(n-1,n-1)$-form on $X$. Set 
$\omega_t^{n-1}=(1-t)\omega^{n-1}_0+t\beta^{n-1}$ and 
$f(t)=(\theta,\omega_t^{n-1})$. Then, there exists $\varepsilon>0$ 
such that $\omega_t^{n-1}\in V$ for $-\varepsilon\leq t\leq \varepsilon$. 
Therefore $f(-\varepsilon)\geq 0$, $f(\varepsilon)\geq 0$, $f(0)=0$, 
and it follows that $f\equiv0,$ 
and so $(\theta,\beta^{n-1})=0$, 
$\forall \beta^{n-1}\in {\mathcal C}^{\infty}_{n-1,n-1}(X, {\mathbb R})$, 
with $d\beta^{n-1}=0$. The duality between 
$H^{n-1,n-1}_{BC}(X, {\mathbb R})$ and $H^{1,1}_A(X, {\mathbb R})$ 
implies that $\{\theta\}=0\in {\mathcal E}_{X,A}^1$.

We can suppose now that $(\theta,\omega^{n-1})>0$, $\forall\omega^{n-1}\in V$. Set 
\begin{align*}
U=&~\{\lambda^{n-1}\in {\mathcal C}^{\infty}_{n-1,n-1}(X, {\mathbb R})\vert\lambda^{n-1} >0\}\\ 
E=&~\{\lambda^{n-1}\in {\mathcal C}^{\infty}_{n-1,n-1}(X, {\mathbb R})\vert d\lambda^{n-1}=0\}\\ 
F=&~\{\lambda^{n-1}\in E\vert (\theta,\lambda^{n-1})=0\}.
\end{align*} 
Then $U\cap E=V$ and $V\cap F=\emptyset,$ and hence $U\cap F=\emptyset$. 
By the Hahn-Banach theorem, we can separate $U$ and $F$ by a current 
$T$ of bidegree $(1,1)$ which is strictly positive on $U$ and vanishes on 
$F$. Then $T$ is a positive current. Let $\omega^{n-1}\in V$ and define 
$\displaystyle \lambda=\frac{(\theta,\omega^{n-1})}{(T,\omega^{n-1})}$. 
Then $\theta-\lambda T$ is zero on $E$ and from the duality between 
$H^{1,1}_A(X, {\mathbb R})$ and $H^{n-1,n-1}_{BC}(X, {\mathbb R})$ 
it follows that there exists $S$ a $(1,0)$-current on $X$ such that 
$$
\theta-\lambda T=-\bar\partial S-\partial\bar S.
$$ 
Hence, the current $\theta+\bar\partial S+\partial\bar S$ is positive and 
$\{\theta\}\in {\mathcal E}_{X,A}^1$.
\end{itemize}
\end{proof}

Let 
\begin{equation*}
\label{Kc}
\cal K_X=\{[\omega]\in H^{1,1}_{BC}(X,\RR)| ~\omega~\text{is a K\"ahler metric}\}
\end{equation*}
denote the K\"ahler cone of $X.$ Similarly, we define the balanced cone:
\begin{equation*}
\label{Kbc}
\cal B_X=\{[\omega]\in H^{n-1,n-1}_{BC}(X,\RR)| ~
\omega^{n-1}~\text{is a balanced metric}\}.
\end{equation*}

\begin{lemma}
\label{kah-bal-cones}
Let X be a compact complex manifold of dimension $n$.

\begin{itemize} 
\item[ i)] If $X$ is K\"ahler, then ${\mathcal N}_{X,BC}^1=\overline{\cal K}_X$. 
Moreover, ${\mathcal E}_{X,A}^{n-1}$
 is closed and we have,  
${\mathcal N}_{X,A}^{n-1}\subset {\mathcal E}^{n-1}_{X,A}.$
\item[ ii)] If $X$ is balanced, then ${\mathcal N}_{X,BC}^{n-1}=\overline{\cal B}_X$. 
Moreover, ${\mathcal E}_{X,A}^1$ is closed.
\end{itemize}
\end{lemma}

\begin{proof} The proof is an adaptation of the arguments in Lemma \ref{E1BC-closed}.

\begin{itemize}
\item [ i)] Since ${\cal K_X}\subset {\mathcal N}_{X,BC}^1$ and  
${\mathcal N}_{X,BC}^1$ is closed,
we can see that $\overline{\cal K }_X\subset {\mathcal N}_{X,BC}^1$. 
Conversely, fix $\omega$ a K\"ahler metric and let $\eta\in {\mathcal N}_{X,BC}^1$.
Then $\eta+t[\omega]\in {\cal K_X}$ for any $t>0$ and 
$\displaystyle\eta=\lim_{t\to 0}\eta+t[\omega]\in \overline{\cal K}_X$.
This proves that $\overline{\cal K_X}={\mathcal N}_{X,BC}^1.$

As in the proof of Lemma \ref{E1BC-closed},  we show that 
${\mathcal E}_{X,A}^{n-1}$ is closed and 
${\mathcal N}_{X,A}^{n-1}\subset {\mathcal E}^{n-1}_{X,A}$. 
Let $\omega$ be a K\"ahler metric on $X$ and $\eta \in \overline{{\mathcal E}_{X,A}^{n-1}}.$ 
Let $T_j$ positive $i\partial\bar\partial$-closed currents of 
bidegree $(n-1,n-1)$ such that $\{T_j\}\to\eta$ in $H_A^{n-1,n-1}(X, {\mathbb R}).$ 
 Then the sequence $(\int_XT_j\wedge\omega)_j$ is 
bounded, hence we can extract a subsequence $(T_{j_k})_k$ which is weakly 
convergent to a positive $i\partial\bar\partial$-closed current $T$ and $T\in \eta.$ 
Therefore  we have $\eta\in {\mathcal E}_{X,A}^{n-1}$. 
In order to prove the 
inclusion ${\mathcal N}_{X,A}^{n-1}\subset {\mathcal E}^{n-1}_{X,A}$, let 
$\eta\in {\mathcal N}_{X,A}^{n-1}.$ Then, by definition, 
$\eta+\varepsilon\{\omega^{n-1}\}\in {\mathcal E}_{X,A}^{n-1},$ 
and since ${\mathcal E}_{X,A}^{n-1}$ is closed, it follows that 
$\displaystyle \eta=\lim_{\varepsilon\to 0}\eta+\varepsilon\{\omega^{n-1}\}\in {\mathcal E}_{X,A}^{n-1}$.

\item [ ii)] We have ${\cal B_X}\subset {\mathcal N}_{X,BC}^{n-1}$ and, since 
${\mathcal N}_{X,BC}^{n-1}$ is closed, it follows that 
$\overline{\cal B}_X\subset {\mathcal N}_{X,BC}^{n-1}$. 
Conversely, fix $\omega^{n-1}$ a balanced metric on $X$ and let 
$\eta\in {\mathcal N}_{X,BC}^{n-1}$. 
Then $\eta+t[\omega^{n-1}]\in {\cal B_X}$ 
for any $t>0$ and therefore 
$\displaystyle \eta=\lim_{t\to 0}\eta+t[\omega^{n-1}]\in \overline{\cal B}_X$.

We show next that ${\mathcal E}_{X,A}^1$ is closed. 
Let $\omega^{n-1}$ be a fixed balanced metric on $X$ and
consider a sequence $(S_j)_j$ of  positive $i\partial\bar\partial$-closed 
currents of bidegree $(1,1)$ converging to 
$\eta\in H^{1,1}_A(X, {\mathbb R}).$
Then the sequence $(\int_XS_j\wedge\omega^{n-1})_j$ is bounded and so
there exists a subsequence $(S_{j_k})_k$ converging weakly to a 
positive $i\partial\bar\partial$-closed current $S$ of 
bidegree $(1,1).$ That means $\{S\}=\eta\in {\mathcal E}_{X,A}^1$, 
and so the cone ${\mathcal E}_{X,A}^1$ is closed. 
\end{itemize}
\end{proof}

\begin{rmk}
If $X$ is a K\a"ahler manifold, there exists a natural map
 $\varpi: \cal K\ra \cal B$ given by 
$\varpi([\omega])=[\omega^{n-1}].$ Fu and Xiao \cite{fx} showed 
that the map $p$ is injective \cite[Proposition 1.1]{fx}. Moreover, 
$p$ is not always surjective. More precisely, they provided 
examples of manifolds \cite[pages 11 and 12]{fx} 
where $\cal B_X\setminus \varpi(\cal K_X)\neq \emptyset.$
\end{rmk}

\bigskip

We have natural morphisms 
\begin{equation*}
j_1:H^{1,1}_{BC}(X, {\mathbb R})\to H^{1,1}_A(X, {\mathbb R})
\end{equation*}
and
\begin{equation*}
j_{n-1}:H^{n-1,n-1}_{BC}(X,{\mathbb R})\to H^{n-1,n-1}_A(X, {\mathbb R})
\end{equation*} 
which are isomorphisms if $X$ is K\"ahler, due to the 
$\partial\bar\partial$-lemma.

\medskip

 \begin{prop}
 Let $X$ be a compact K\"ahler manifold of dimension $n$. Then 
\begin{equation}
j_{n-1}({\mathcal E}_{X,BC}^{n-1})={\mathcal E}_{X,A}^{n-1}
\end{equation}
 and 
\begin{equation}
j_1({\mathcal N}_{X,BC}^1)={\mathcal N}_{X,A}^1
 \end{equation}
 \end{prop}

\begin{proof} The second statement follows from the first 
one by duality. From Theorem \ref{duality}, we have that 
$({\mathcal N}_{X,BC}^1)^*={\mathcal E}_{X,A}^{n-1}$ 
since ${\mathcal E}_{X,A}^{n-1}$ is closed. Corollary 0.3 
in \cite{dempau} implies that the currents of the form 
$j_{n-1}([\int_Y\omega^{p-1}\wedge \bullet ]),$ where $Y$ 
is a $p$-dimensional analytic subset of $X$ and $\omega$ 
is a K\"ahler metric on $X,$ generate the cone 
${\mathcal E}_{X,A}^{n-1}$. Since the currents 
$\left(\int_Y\omega^{p-1}\wedge \bullet\right)$ are 
$d$-closed and positive, we see that 
$j_{n-1}({\mathcal E}_{X,BC}^{n-1})={\mathcal E}_{X,A}^{n-1}.$ 
\end{proof}

\begin{rmk}
Given a compact complex K\"ahler manifold of dimension $n,$ 
Conjecture 2.3 in \cite{boucksom} implies that
$j_{n-1}({\mathcal N}_{X,BC}^{n-1})={\mathcal N}_{X,A}^{n-1},$ 
i.e., that the dual of the pseudoeffective cone 
${\mathcal E}_{X,BC}^1$ is the closure of the cone of classes of balanced 
metrics.
\end{rmk}

\subsection{N\'eron-Severi groups}
\label{nsg}

For a compact complex manifold $X$ of dimension $n$ 
we have natural maps 
\begin{equation*}
\alpha_p:H_{BC}^{p,p}(X,{\mathbb R})\to H^{2p}_{dR}(X,{\mathbb R}),
\end{equation*}
\begin{equation*}
\beta_p:H^{2p}_{dR}(X,{\mathbb R})\to H^{p,p}_A(X,{\mathbb R}),
\end{equation*}
\begin{equation*} 
\gamma_p:H^{2p}(X,{\mathbb Z})\to H^{2p}_{dR}(X,{\mathbb R}).
\end{equation*}

Define the N\'eron-Severi groups
\begin{equation*} 
H^{p,p}_{BC,NS}(X,{\mathbb R})=
\alpha_p^{-1}(\gamma_p(H^{2p}(X,{\mathbb Z})))
\otimes_{\mathbb Z}{\mathbb R}\subset H^{p,p}_{BC}(X,{\mathbb R})
\end{equation*}
 and 
\begin{equation*}
H^{p,p}_{A,NS}(X,{\mathbb R})=
\beta_p(\gamma_p(H^{2p}(X,{\mathbb Z})))
\otimes_{\mathbb Z}{\mathbb R}\subset H^{p,p}_A(X,{\mathbb R}).
\end{equation*}

If $X$ is projective, then the canonical morphisms 
\begin{equation*}
H^{1,1}_{BC,NS}(X,{\mathbb R})\to H^{1,1}_{A,NS}(X,{\mathbb R})
\end{equation*}
and
\begin{equation*}
H^{n-1,n-1}_{BC,NS}(X,{\mathbb R})\to H^{n-1,n-1}_{A,NS}(X,{\mathbb R}).
\end{equation*}
are isomorphisms, and the standard notation for these groups 
are $N^1$ or $NS^1_X,$ and $N_1,$ respectively. The group $NS^1_X$ 
is generated by classes of divisors on $X$, and by the Hard Lefschetz 
Theorem, it follows that $N_1$ is generated by classes of curves on $X$.

\medskip

Let the subscript $NS$ denote the intersection of a cone 
(nef or pseudoeffective) with the N\'eron-Severi groups.

 \begin{prop}
 \label{neronseveri}
 If $X$ is compact K\"ahler of dimension $n$, then the pairing 
\begin{equation}
\label{pairing}
H^{p,p}_{BC,NS}(X,{\mathbb R})\times 
H^{n-p,n-p}_{A,NS}(X,{\mathbb R})\to {\mathbb R}, 
([\alpha],\{\beta\})\to\int_X\alpha\wedge\beta
\end{equation} 
is nondegenerate and all the equalities of Theorem \ref{duality}
hold at the N\'eron-Severi level. Moreover, 
\begin{equation*}
j_{n-1}({\mathcal E}^{n-1}_{BC,NS})={\mathcal E}^{n-1}_{A,NS}
\end{equation*}
 and 
\begin{equation*}
j_1({\mathcal N}^1_{BC,NS})={\mathcal N}^1_{A,NS}.
\end{equation*}
 If $X$ is projective, then 
\begin{equation}
\label{2.11}
j_{n-1}({\mathcal N}_{BC,NS}^{n-1})={\mathcal N}_{A,NS}^{n-1}
\end{equation}
 and 
\begin{equation}
\label{2.12}
j_1({\mathcal E}_{BC,NS}^1)={\mathcal E}_{A,NS}^1.
 \end{equation}
 \end{prop}
 \begin{proof}
The only non-trivial statement is (\ref{2.12}), as  (\ref{2.11})  
follows by duality.

Let $\{T\}\in {\mathcal E}_{A,NS}^1$ where $T$ is a positive, 
$\partial\bar\partial$-closed current, and let 
$j_1([S])=\{T\}$, $[S]\in H^{1,1}_{BC,NS}(X,{\mathbb R})$. 
We want to show that $[S]\in {\mathcal E}_{BC,NS}^1.$ For the 
proof, we follow \cite{toma}.

From \cite[Theorem 2.2]{boucksom} we see
that it is enough to check that 
$$
([S],\{p_*(A_1\cap \ldots \cap A_{n-1})\})\geq 0,
$$
where $p:Y\to X$ is a proper modification of $X$ and 
$A_1,\ldots,A_{n-1}$ are  
very ample line bundles on $Y$. 
However, from Theorem $3$ in \cite{alessandrini2}, there exists $T'$, 
a positive pluriharmonic current on $Y$ which is the total transform 
of $T$, and we have
\begin{align*}
([S],\{p_*(A_1\cap\ldots\cap A_{n-1})\})&=(\{T\},[p_*(A_1\cap\ldots\cap A_{n-1}])\\
&=(T',A_1\cap\ldots \cap A_{n-1})\\
&\geq 0.
\end{align*}
\end{proof}

\begin{rmk}
\label{psefbcdra}
Boucksom, Demailly, P\u aun and Peternell define 
\cite[Definition 1.1]{boucksom} the pseudoeffective cone 
${\mathcal E}_{NS}$ as ${\cal E}_{X, dR}^{1}
\cap NS_{\mathbb R}(X)$, 
where  
$$
NS_{\mathbb R}(X)=(H^{1,1}_{\mathbb R}(X)\cap 
H^2(X, {\mathbb Z})/{\rm torsion})
\otimes_{\mathbb Z}{\mathbb R}.
$$
Formula (\ref{2.12}) above implies in particular that, 
at the N\'eron-Severi level, the pseudoeffective cones 
${\mathcal E}_{BC,NS}^1,~{\mathcal E}_{A,NS}^1$ and 
${\mathcal E}_{NS}$ coincide via the canonical isomorphisms 
between the cohomology groups $H^{1,1}_{BC,NS}
(X, {\mathbb R}),~ 
H^{1,1}_{A,NS}(X, {\mathbb R})$, 
and $NS_{\mathbb R}(X)$. 
\end{rmk}

\section{Uniruled manifolds and balanced metrics}
\label{metrics}

\subsection{Bimeromorphism invariance}
\label{bimero}

We prove here that the existence of a balanced metric of 
positive total scalar Chern curvature is an invariant property 
under bimeromorphisms.

\medskip

\begin{proof}[Proof of Theorem \ref{bi-invariance}]
By \cite{wft}, we can assume that $p:Y\to X$ is a blow-up with 
smooth center $C$ and let $E$ be the exceptional divisor of $p.$ 
Then $K_Y=p^*K_X+aE$, where $a=\codim_XC-1>0.$

\medskip

Suppose first that $X$ admits a balanced metric $\omega_X^{n-1}$ 
which is negative on the canonical line bundle of $X$. 
Let $i:E\to Y$ denote the inclusion. Since 
$$
\int_Y c_1(E)\wedge p^*\omega_X^{n-1}=\int_Ei^*p^*\omega_X^{n-1}=\int_C\omega_X^{n-1}=0,
$$ 
we find that  
$$
\int_Yc_1(K_Y)\wedge p^*\omega_X^{n-1}
=\int_Yc_1(p^*K_X)\wedge p^*\omega_X^{n-1}
=\int_Xc_1(K_X)\wedge \omega_X^{n-1}<0.
$$ 
It is known that $Y$ is also balanced \cite{alessandrini4}, 
and if $\omega_Y^{n-1}$ is a balanced metric on $Y$, then 
$p^*\omega_X^{n-1}+\varepsilon \omega_Y^{n-1}$ is a 
balanced metric and 
$$
\int_Yc_1(K_Y)\wedge (p^*\omega_X^{n-1}+\varepsilon \omega_Y^{n-1})<0,
$$ 
for a small $\varepsilon >0$.

Conversely, suppose that $Y$ supports a balanced metric 
$\omega_Y^{n-1}$ such that 
\begin{equation}
\label{negativeY}
\int_Yc_1(K_Y)\wedge \omega_Y^{n-1}<0.
\end{equation}
and suppose that
\begin{equation*}
\int_Xc_1(K_X)\wedge\omega^{n-1}\geq 0
\end{equation*}
for any balanced metric $\omega^{n-1}$ on $X$. Then
\begin{equation*}
\int_Xc_1(K_X)\wedge \eta\geq 0
\end{equation*}
for any class $[\eta]\in {\cal N}_{BC,X}^{n-1}$. Therefore, 
by Theorem \ref{duality} iv), $\{c_1(K_X)\}\in {\cal E}_{X,A}^1,$ i.e., 
there exists $T$ a positive $\partial\bar\partial$-closed 
$(1,1)$-current in the Aeppli cohomology class $\{c_1(K_X)\}$. 
From \cite{alessandrini2}, it follows that there exists  
a positive $\partial\bar\partial$-closed current on $Y$ denoted 
by $T',$ which is the total transform of $T.$ This means that 
$$
T'\in \{c_1(p^*K_X)\}=p^*\{c_1(K_X)\}.
$$ 
In particular, 
$\{c_1(p^*K_X)\}\in {\cal E}_{Y,A}^1$ and therefore 
$$
\{c_1(K_Y)\}=\{c_1(p^*K_X)\}+a\{[E]\}\in {\cal E}_{Y,A}^1
$$ 
which contradicts (\ref{negativeY}).
\end{proof}

\subsection{Metrics on Mori fiber spaces}
\label{mfs}

We start by recalling background definitions from the minimal 
model program. 

\medskip

\begin{defn}
A compact complex variety $Y$ is called $\QQ-$factorial if 
every Weil divisor of $Y$ is $\QQ-$Cartier.
\end{defn}

Let $Y$ be normal variety such that $mK_Y$ is Cartier for some 
$m>0,$ and let $f:Z\ra Y$ be a resolution of singularities.  Up to
numerical equivalence, we can write 
$$
K_Z\equiv_{\QQ}f^*(K_Y)+\sum_i a_iE_i,
$$
where the $E_i$'s are the $f-$exceptional divisors, and $a_i\in \QQ.$

\begin{defn}
We say that $Y$ has log-terminal singularities if $a_i>-1,$ for all $i.$
\end{defn}
It is well-known that this definition is independent of the choice of 
the resolution \cite{km}.

\begin{defn}
A normal compact complex variety $Y$ with only $\QQ-$factorial 
log-terminal singularities equipped with a map $\phi:Y\ra B$ is 
called a Mori fiber space if the following conditions are satisfied:
\begin{itemize}
\item[ i)] The map $\phi$ is a morphism with connected fibers onto 
a normal variety $B$ with $\dim B<\dim Y.$

\item[ ii)] All the curves $C$ in the fibers of $\phi$ are numerically 
proportional and $K_Y\cdot C<0.$
\end{itemize}
\end{defn}

\subsubsection{The projective case}

We give here a first proof of Theorem \ref{thmA-proj} based on the 
the minimal model program. A second proof, circumventing the 
minimal model program follows.

\begin{prop}
\label{metric-mfs}
Let $\phi:Y\ra B$ be a Mori fiber space, with $Y$ and $B$ projective. 
Then, there exists an ample line bundle $H$ on $Y$ such that 
$$
K_Y\cdot H_Y^{n-1}<0.
$$
\end{prop}

\begin{proof} If $\dim B=0,$ by Kleiman's Ampleness Criterion $-K_Y$ 
is ample, and so $K_Y\cdot H^{n-1}<0$ for all ample line bundles on $Y.$

Assume now that $\dim B=b>0$ and fix an ample line bundle 
$L$ on $B,$ and $H_0$ an ample line bundle on $Y.$ Let 
$$
H_m=m\phi^*L+H_0.
$$
Then $H_m$ is an ample line bundle on $Y$ for all $m>0,$ and 
\begin{align*}
K_Y\cdot H_m^{n-1}~=~&K_Y\cdot (m\phi^*L+H_0)^{n-1}\\
~=~&c(n,b)m^bK_Y\cdot (\phi^*L)^b\cdot H_0^{n-1-b} +O(m^{b-1})\\
~=~&c(n,b)m^b(L^b)(K_F\cdot H_0^{n-1-b}) +O(m^{b-1}),
\end{align*}
where $c(n,b)$ is a positive integer depending only on $n$ and $b,$ 
and $F$ denotes the fiber of $\phi.$ By the relative version of 
Kleiman's Ampleness Criterion \cite[Theorem 1.44]{km} 
we see that $-K_F$ is ample, and so 
$K_Y\cdot H_m^{n-1}<0$ for $m\gg0.$ Take now $H_Y=H_m$ 
for some fixed $m\gg 0.$
\end{proof}

\begin{proof}[The first proof of Theorem \ref{thmA-proj}]
Let $X$ be a smooth, Moishezon, uniruled manifold of dimension 
$n>0.$ Then there exists a smooth projective manifold $Y$ of 
dimension $n$ bimeromorphic to $X.$ Since uniruledness is preserved 
under bimeromorphic transformations, $Y$ is uniruled.
According to \cite[Theorem IV.1.9]{kollar}, there exists 
a non-constant holomorphic map $u:\PP_1\ra Y,$ such 
that $u^*T_Y$ is globally generated. Since we have an injection 
from $\OO_{\PP^1}(2)=T_{\PP^1}$ to $u^*T_Y,$ it follows that 
$\deg u^*T_Y\geq 2,$ and so $K_Y\cdot u(\PP_1)<0.$ But 
the curve $u(\PP_1)$ moves in a family covering $X,$ and so  
by \cite[Theorem 0.2]{boucksom},  the canonical bundle 
$K_Y$ is not pseudoeffective. This implies, according to 
\cite[Corollary 1.3.3]{bchm}, that  $Y$ is birational to a Mori fiber space 
$\phi: Z\ra B$ with $Z$ and $B$ projective. In general, $Z$ is not 
smooth, and let $f:\hat Z\ra Z$ be a desingularization. Then, there 
exists an ample line bundle $H_{\hat Z}$ on ${\hat Z}$ such that 
$$
K_{\hat Z}\cdot H_{\hat Z}^{n-1}<0.
$$
Indeed, from Proposition \ref{metric-mfs}, we know that 
there exists an ample line bundle $H_Z$ on $Z$ such that 
$K_Z\cdot H_Z^{n-1}<0.$ Fix $H_0$ be an ample line bundle 
on ${\hat Z}.$ For every $m>0,$ let 
$$
H_m=mf^*H_Z+H_0
$$ 
Then $H_m$ is an ample line bundle on $\hat Z,$ and 
\begin{align*}
K_{\hat Z}\cdot H_m^{n-1}~=
~&~(f^*K_Z+\sum_i a_iE_i)\cdot (mf^*H_Z+H_0)^{n-1}\\
~=~&m^{n-1}K_Z\cdot H_Z^{n-1}+O(m^{n-2})<0,
\end{align*}
for $m$ sufficiently large. Take now 
$H_{\hat Z}=H_m$ for fixed $m\gg 0.$ 
Since $H_{\hat Z}$ is ample, the first Chern class 
of $kH_m$ is represented by the K\"ahler form of 
a Hodge metric $\omega_{\hat Z}$ for sufficiently 
large $k.$ In particular, we found on $\hat Z$ a K\"ahler 
metric $\omega$ such that 
$$
\int_{\hat Z} c_1(K_{\hat Z})\wedge\omega_{\hat Z}^{n-1}<0.
$$ 
Since ${\hat Z}$ and $X$ are smooth and bimeromorphic
manifolds, we can apply now Theorem \ref{bi-invariance} 
to conclude that $X$ admits a balanced metric $\omega^{n-1}$ 
such that 
$$
\int_X c_1(K_X)\wedge\omega^{n-1}<0.
$$ 
\end{proof}

\medskip

We give next a short second proof of Theorem \ref{thmA-proj} 
using the results presented in Section \ref{cones}. In fact, for 
projective manifolds we can prove a slightly more precise result:

\medskip

\begin{prop}
\label{projective}
Let $X$ be a uniruled projective manifold of dimension $n$. 
Then there exists $\omega^{n-1}$ a balanced metric, 
$[\omega^{n-1}]\in H^{n-1,n-1}_{BC,NS}(X,{\mathbb R})$ such that 
\begin{equation}
\int_X c_1(K_X)\wedge\omega^{n-1}<0 
\end{equation}
\end{prop}

\begin{proof} Since $X$ is uniruled, arguing as in the first proof of Theorem 
\ref{thmA-proj}, we see that the canonical bundle 
$K_X$ is not pseudoeffective. Hence, as in Remark \ref{psefbcdra},  
$c_1(K_X)\notin {\mathcal E}^1_{BC,NS}.$ As a consequence, from  
Proposition \ref{neronseveri} we see $c_1(K_X)\notin {\mathcal E}^1_{A,NS}.$ 
Theorem \ref{duality} now implies the existence of a balanced metric 
$\omega^{n-1}$ with integral class whose pairing with $c_1(K_X)$ is 
negative.
\end{proof}

The proof of Theorem \ref{thmA-proj} now follows from Proposition 
\ref{projective} and Theorem \ref{bi-invariance}.

\subsubsection{The K\"ahler case} 

The first proof of Therem \ref{thmA-proj} can be adapted 
in K\"ahler setting.\footnote{Here we have to work with 
singular K\"ahler spaces. For the basic notions in 
the theory of a K\"ahler space we refer the interested reader 
to the sections 2 and 3 in \cite{hor-pet2}.}

\medskip

\begin{prop}
\label{hp-kmfs}
Let $\phi:Z\ra S$ be a Mori fiber space where $Z$ and $S$ are 
K\"ahler spaces. Then there exists a K\"ahler form $\eta$ on $Z$ 
such that 
$$
K_Z\cdot[\eta^{n-1}]<0.
$$
\end{prop} 

\begin{proof} 
Fix $\omega_S$  and $\omega_Z$ K\"ahler forms 
on $S$ and $Z,$ respectively and consider the 
family of K\"ahler forms  
$$
\eta_t=t\phi^*\omega_S+\omega_Z, ~t>0.
$$
As in the proof of Proposition \ref{metric-mfs} we see that 
$K_Z\cdot [\eta_t^{n-1}]<0$ for $t\gg0.$ We omit 
the details. 
\end{proof}

\begin{proof}[Proof of Theorem \ref{thmA}] Let $X$ be a smooth, 
uniruled,  $3-$dimensional manifold of class $\cal C.$ 
That means there exists a  uniruled,  $3-$dimensional, K\"ahler  
manifold $Y$ bimeromorphic to $X.$

According to H\"oring and Peternell \cite[Theorem 1.1]{hor-pet2}, 
$Y$ is bimeromorphic to a K\"ahler Mori fiber space $Z$ as 
in Proposition \ref{hp-kmfs}. In general, $Z$ is not smooth. 
Let $f:\hat Z\ra Z$ be a desingularization. By \cite[1.3.1]{varouchas},  
$\hat Z$ is a smooth K\"ahler manifold. As in the first proof of 
Theorem \ref{thmA-proj}, we can find a K\"ahler metric $\omega$ on $\hat Z$ 
such that 
$$
\int_Yc_1(K_{\hat Z})\wedge \omega^2<0.
$$ 
Applying now Theorem \ref{bi-invariance}, we can 
conclude that $X$ admits a balanced 
metric with the property claimed in Theorem 
\ref{thmA}.
\end{proof}

\subsection{Characterization of uniruledness}
\label{uniruledness}

In this very short section, we complete the characterization of uniruledness, by 
proving a converse to Theorems \ref{thmA-proj} and \ref{thmA}.

 \begin{proof}[Proof of Theorem \ref{characterization}] 
 The implication $ii)\Longrightarrow iii)$ is the content of Theorems  
 \ref{thmA-proj} and \ref{thmA}, while $iii)\Longrightarrow iv)$ is 
 trivial. Morever, from the positivity criterion of Lamari 
 \cite[Th\'eor\`eme 1.2 (1)]{lamari1} (see also Theorem \ref{duality}, part i)) 
we can see that  $iv)\Longrightarrow i)$.

Finally, it remains to show that $i)\Longrightarrow ii)$. Since neither 
uniruledness nor the pseudoeffectivity of the canonical divisor is affected 
by bimeromorphic transformations, we may assume that either $X$ is 
projective, or  $X$  is a non-projective K\"ahler threefold. In the first 
case, the remarkable Corollary 0.3 in \cite{boucksom} shows that $X$ is 
uniruled, while in the second case we reach the same conclusion using 
the equally remarkable Corollary 1.2 in \cite{brunella}.
 \end{proof}

\section{Balanced metrics on twistor spaces}
\label{sec-twist}

A large class of examples of uniruled complex manifolds is provided 
by the manifolds bimeromorphic to the twistor spaces of closed anti-self 
dual four-manifolds. These are compact complex manifolds of dimension 
three \cite{ahs}, equipped with a one-parameter family of balanced 
metrics  \cite{michelsohn, muskarov}.
In this section, we show that among these metrics there exists a  
balanced metric of positive total Chern  scalar curvature.  

\medskip

We start by recalling the construction of the twistors spaces.

\bigskip

Let $(M,g)$ be an oriented Riemannian $4-$manifold. 
Under the action of the Hodge $\star-$operator 
$$
\star:\Lambda^2M\to \Lambda^2 M,
$$  
one has a decomposition 
$ \displaystyle 
\Lambda^2M=\Lambda_+\oplus \Lambda_-
$ 
into self-dual and anti-self-dual forms, 
corresponding to the $(\pm 1)-$ eigenvalues of $\star.$

Let $\cal R: \Lambda^2\to\Lambda^2$ be the Riemannian
curvature operator. Under the action of $SO(4),$ 
the Riemannian curvature operator 
decomposes as 
$$
\cal R=\frac{s}6Id+W^-+W^++\stackrel{\circ}{r},
$$ 
where $s$ denotes the scalar curvature, $W^{\pm}$ 
are the self-dual and anti-self-dual components of the 
Weyl curvature operator, and $\stackrel{\circ}{r}$ 
is the trace-free Ricci curvature operator.  The oriented Riemannian 
$4-$manifold $(M, g)$ is said to be anti-self-dual (ASD) 
if $W^+=0.$ This definition is conformally invariant, i.e. 
if $g$ is ASD, so is $ag$  for any smooth positive 
function $a.$

The twistor space of a conformal Riemannian manifold 
$(M,[g])$ is the total space of the sphere 
bundle of the rank three real vector bundle of self-dual 
$2-$forms ${\cal Z}:=S(\Lambda_+).$ 
Let $\varpi:{\cal Z}\to M$ be the projection onto $M.$ For every 
$x\in M,$ the fiber $\varpi^{-1}(x)$ corresponds to the set of 
$g-$orthogonal complex structures compatible with the given 
orientation. More precisely, any such $j$ defines the unit 
length self-dual form 
$$
\omega_{\cal j}(v,w)=\frac1{\sqrt 2}g(v,jw).
$$

The real six-dimensional manifold $\cal Z$ comes equipped with 
an almost complex structure, that is an endomorphism 
${\cal J}:T{\cal Z}\to T{\cal Z}$ satisfying ${\cal J}^2=-1.$  
The Levi-Civita connection $\nabla$ of $M$ gives rise to a splitting
$T{\cal Z} ={\cal H}\oplus {\cal V}$ of the tangent bundle of ${\cal Z}$ 
into horizontal and vertical components. At a point $(\sigma, x)\in \cal Z,$ 
the vertical distribution $\cal V$ consists of the vectors tangent to 
the fiber of $\varpi,$ which is an oriented metric $2-$sphere, 
and hence equipped with a  compatible complex structure $I.$ 
On the other hand, the almost  complex structure $j$ associated 
to $\sigma$ discussed above naturally lifts to the horizontal distribution 
$\cal H.$ Then, $\cal J$ is defined as $\cal J=(j,I).$ A remarkable result 
of Atiyah, Hitchin and Singer \cite{ahs} asserts that $\cal J$ is integrable 
if and only if the metric $g$ is ASD. In such a case, the fibers 
$\varpi^{-1}(x), x\in M$ are smooth rational curves, and so $\cal Z$ 
is uniruled.

\bigskip

We assume from now on that $\mathcal Z$ is the twistor space 
associated to a closed, oriented $4$-manifold $M$ equipped with an ASD 
conformal class $[g]$. We fix $g\in[g]$.

Let $h_t$ be the family of Riemannian metrics on $\cal Z$ 
defined by 
\begin{equation}
\label{bmetrics}
h_t=\varpi^*g+tg^{\text{vert}},
\end{equation}
where $t>0,$ $g$ is the metric of $M$ and $g^{\text{vert}}$ 
is the restriction of the metric induced on $\Lambda_+$ 
to the vertical distribution $\cal V.$ Then 
$\varpi:(\cal Z, h_t)\to (M,g)$ is a Riemannian submersion 
with totally geodesic fibers. Moreover, the metrics $h_t$ 
are compatible with $\cal J.$ Michelsohn states 
\cite[Section 6]{michelsohn} the existence of balanced metrics on 
$\mathcal Z.$ A proof that the metrics $h_t$ are in fact balanced 
follows from Corollary 3.5 and Lemma 4.1 in \cite{muskarov}.

The Riemannian scalar curvature of the metrics $h_t$ 
is computed by  Davidov and Mu\v skarov \cite{dm}. 
More precisely, in \cite[Corollary 4.2]{dm} it is proved 
that for every $(\sigma, x)\in \cal Z,$  
$$
s_{\cal Z}(\sigma, x)= 
s_M(x)+\frac t4(\|\cal R(\sigma)\|^2-\|\cal R_-\|_x^2)+\frac 2t,
$$
where $s_{\cal Z}$ and $s_M$ denote the scalar 
curvatures of $\cal Z$ and $M,$ respectively, and 
$\cal R_-=\frac{s}{12}Id+W^-+\stackrel{\circ}{r}$ 
is the restriction of $\cal R$ to $\Lambda_-.$
In particular, for $0<t\ll 1,$ we see that the metric $h_t$ 
satisfies $s_{\cal Z}>0.$

\begin{proof}[Proof of Theorem \ref{twistors}] Let $X$ 
be a complex manifold bimeromorphic to a twistor 
space $\cal Z.$

Let $\omega_t$ be the K\"ahler 2-form of the 
balanced metric $h_t$ on $\cal Z$ defined by (\ref{bmetrics}).
By Corollary \ref{ineq}, we have 
$$
\int_{\cal Z} c_1({\cal Z})\wedge \omega_t^2\geq 
\frac 1{12\pi}\int_{\cal Z}s_{\cal Z}\omega_t^3>0,
$$
for $0<t\ll 1.$
The conclusion of Theorem \ref{twistors} now 
follows from Theorem \ref{bi-invariance}.
\end{proof}

\begin{rmk}
\label{ex-twistors} 
Twistor spaces of class $\cal C$ are rather scarce.
Campana \cite{campana}, and LeBrun and Poon \cite{lp}, 
independently proved that if the twistor space  $\cal Z$ of an 
ASD four-manifold $M$ is of Fujiki class $\cal C,$ then $\cal Z$ 
is Moishezon and $M$ is homeomorphic to either $S^4$ 
or the connected sum of $n\geq 1$ copies of $\cpb,$ the 
complex projective plane endowed with the opposite 
orientation. However, a result of Taubes \cite{taubes} 
asserts that every Riemannian manifold $M$ can be 
equipped with an ASD metric after taking the connected 
sum with sufficiently many copies of $\cpb$, hence the 
twistor spaces provide a large family of balanced manifolds 
which are not of class $\cal C.$ 
\end{rmk}

\subsection*{Acknowledgements} The first author was supported by the 
CNCS grant PN-II-ID-PCE-2011-3-0269 during the preparation of this work. 
The second author acknowledges the support of the Simons Foundation's 
"Collaboration Grant for Mathematicians", while the third author was supported 
by the NSF grant DMS-1309029. The second and third author would like to 
thank IH\'ES for hospitality, while this project was finalized. The authors 
are grateful to the anonymous referees  for many useful comments
which helped us improve this article.

 \providecommand{\bysame}{\leavevmode\hbox
to3em{\hrulefill}\thinspace}

\end{document}